\documentclass[12pt]{amsart}
\usepackage{amsmath,amssymb,amsthm,amscd}
\newtheorem{formula}{}[section]
\newtheorem{proposition}[formula]{Proposition}
\newtheorem{corollary}[formula]{Corollary}

\newtheorem{theorem}[formula]{Theorem}
\theoremstyle{definition}
\newtheorem{definition}[formula]{Definition}
\newtheorem{example}[formula]{Example}
\theoremstyle{remark}
\newtheorem*{remark}{Remark}

\begin{document}
\title[Adjoint Cohomology]
{Adjoint cohomology of graded Lie algebras of maximal class}
\author{Dmitri V. Millionschikov}
\thanks{Partially supported by
the Russian Foundation for Fundamental Research, grant no.
05-01-01032 and Scietific Schools 2185.2003.1} 
\address{Department of Mathematics and Mechanics, Moscow
State University, 119899 Moscow, RUSSIA}

\begin{abstract}
We compute explicitly the adjoint cohomology of two ${\mathbb
N}$-graded Lie algebras of maximal class (infinite dimensional
filiform Lie algebras) ${\mathfrak m}_0$ and ${\mathfrak m}_2$. It
is known that up to an isomorphism there are only three ${\mathbb
N}$-graded Lie algebras of the maximal class. The third algebra
from this list is  the "positive" part $L_1$ of the Witt (or
Virasoro) algebra and its adjoint cohomology was computed earlier
by Feigin and Fukhs. We show that the total space $H^*({\mathfrak
m}_j,{\mathfrak m}_j)$ is "almost" isomorphic to the completed
tensor product ${\mathfrak m}_j \otimes H^*({\mathfrak m}_j)$,
$j=0,2$.
\end{abstract}
\date{}

\maketitle

\section*{Introduction}

A.~Shalev and E.~Zelmanov defined in \cite{ShZ} the {\it coclass}
(which might be infinity) of a finitely generated and residually
nilpotent Lie algebra $\mathfrak{g}$ as $cc(\mathfrak{g})=
\sum_{i\ge 1}(\dim (C^i\mathfrak{g}/C^{i+1}\mathfrak{g})-1)$,
where $C^i\mathfrak{g}$ denotes the $i$-th ideal of the central
descending series of $\mathfrak g$. Algebras of coclass $1$ are
also called algebras of {\it maximal class}. In finite dimensional
case the notion of Lie algebra of maximal class is equivalent to
the notion of {\it filiform} Lie algebra introduced by Vergne in
\cite{V}. In the study of filiform Lie algebras the ${\mathbb
N}$-graded filiform Lie algebra $\mathfrak{m}_0(n)$ plays a
special role. It is defined by its basis $e_1,\dots,e_n$ and its
non-trivial commutator relations: $[e_1,e_i]=e_{i+1},
i=2,\dots,n{-}1$. The natural infinite dimensional analog of
$\mathfrak{m}_0(n)$ is a ${\mathbb N}$-graded Lie algebra
$\mathfrak{m}_0$ of maximal class. It follows from \cite{Fial}
(see also \cite{ShZ}) that up to an isomorphism there are only
three ${\mathbb N}$-graded Lie algebras of maximal class (or
infinite dimensional filiform Lie algebras):
$$
\mathfrak{m}_0, \; \mathfrak{m}_2, \; L_1,
$$
where $L_1$ denotes the "positive" part of the Witt or Virasoro
algebra and $\mathfrak{m}_2$ is defined by its infinte basis
$e_1,e_2,\dots,e_n,\dots$ and the structure relations:
$$
[e_1,e_i]=e_{i+1}, i=2,\dots, \quad [e_2,e_j]=e_{j+2}, j=3,\dots.
$$
We recall that the scalar cohomology $H^*(L_1)$ was calculated in
\cite{G}. The cohomology $H^*(\mathfrak{m}_0)$ and
$H^*(\mathfrak{m}_2)$ were calculated in \cite{FialMill}.

The next natural question is to calculate the cohomology of these
algebras with coefficients in the adjoint representation, which is
very important due to applications in the deformation theory. The
adjoint cohomology $H^*(L_1,L_1)$ was calculated in \cite{FeFu}.
In the present article we explicitly compute the cohomology
$H^*(\mathfrak{m}_0,\mathfrak{m}_0)$ and
$H^*(\mathfrak{m}_2,\mathfrak{m}_2)$. One has to remark that the
first cohomology spaces $H^1(\mathfrak{m}_0(n),\mathfrak{m}_0(n))$
and $H^1(\mathfrak{m}_2(n),\mathfrak{m}_2(n))$ were computed in
\cite{V,Kh} and this result implies the answer in the infinite
dimensional case. The structure of
$H^1(\mathfrak{m}_0,\mathfrak{m}_0)$ and
$H^1(\mathfrak{m}_2,\mathfrak{m}_2)$ was rediscovered recently in
\cite{FialWag, FialWag2}. Vergne's algorithm for constructing of a
basis of $H^2(\mathfrak{m}_0(n),\mathfrak{m}_0(n))$ can be easily
generalized to the infinite dimensional case.

The paper is organized as follows. In Sections 1--2 we review all
necessary definitions and facts concerning Lie algebra cohomology
and Lie algebras of maximal class. We recall a natural filtration
of the adjoint complex $C^*(\mathfrak{g},\mathfrak{g})$ and the
corresponding spectral sequence $E_r$ for an arbitrary  ${\mathbb
N}$-graded Lie algebra $\mathfrak g$. This spectral sequence was
successfully used by Feigin and Fukhs \cite{FeFu} for the
computation of $H^*(L_1,L_1)$. This spectral sequence is the main
technical tool that we use for the computation of
$H^*(\mathfrak{m}_0,\mathfrak{m}_0)$ and
$H^*(\mathfrak{m}_2,\mathfrak{m}_2)$. The first term of this
spectral sequence is isomorphic to the completed tensor product
${\mathfrak g} \otimes H^*({\mathfrak g})$, where $H^*({\mathfrak
g})$ stands for the scalar cohomology of ${\mathfrak g}$. In some
sense we show that the differentials of higher orders of our
spectral sequence are "almost" trivial.

In the sections 3,5 we recall the structure of the cohomology
$H^*(\mathfrak{m}_0)$ and $H^*(\mathfrak{m}_2)$ with trivial
coefficients obtained in \cite{FialMill}.

In the sections 4,6 we apply the spectral sequence considered
above to the computation of $H^*(\mathfrak{m}_0,\mathfrak{m}_0)$
and $H^*(\mathfrak{m}_2,\mathfrak{m}_2)$ respectively. Namely we
give the explicit answer in terms of formal series of the infinite
basis of cocycles $\{\Psi_{I,r}\}$ in the case of $\mathfrak{m}_0$
and $\{\Phi_{J,q}\}$ for the algebra $\mathfrak{m}_2$.

\section{Lie algebras of maximal class and filiform Lie algebras}

The sequence of ideals of a Lie algebra $\mathfrak{g}$
$$C^1\mathfrak{g}=\mathfrak{g} \; \supset \;
C^2\mathfrak{g}=[\mathfrak{g},\mathfrak{g}] \; \supset \; \dots
\; \supset \;
C^{k}\mathfrak{g}=[\mathfrak{g},C^{k-1}\mathfrak{g}] \; \supset
\; \dots$$
is called the descending central sequence of $\mathfrak{g}$.

A Lie algebra $\mathfrak{g}$ is called nilpotent if
there exists $s$ such that:
$$C^{s+1}\mathfrak{g}=[\mathfrak{g}, C^{s}\mathfrak{g}]=0,
\quad C^{s}\mathfrak{g} \: \ne 0.$$ The natural number $s$ is
called the nil-index of the nilpotent Lie algebra $\mathfrak{g}$.

\begin{definition}
A nilpotent $n$-dimensional Lie algebra $\mathfrak{g}$ is  called
filiform Lie algebra if it has the nil-index $s=n-1$.
\end{definition}

\begin{example}
The Lie algebra $\mathfrak{m}_0(n)$ is defined by its basis $e_1,
e_2, \dots, e_n$ with the commutating relations:
$$ [e_1,e_i]=e_{i+1}, \; \forall \; 2\le i \le n{-}1.$$
\end{example}

\begin{proposition}
Let $\mathfrak{g}$ be a filiform Lie algebra, one can remark that
$$\sum_{i\ge 1}(\dim (C^i\mathfrak{g}/C^{i+1}\mathfrak{g})-1)=1,$$
where in the sum only the first summand is non trivial.
\end{proposition}

\begin{definition}
A Lie algebra $\mathfrak{g}$ is called residually nilpotent if
$$\cap_{i=1}^{\infty}C^{s}\mathfrak{g}=0.$$
\end{definition}

\begin{definition}
A coclass of Lie algebra $\mathfrak{g}$ (which might be infinity)
is a number $cc(\mathfrak{g})$ defined as $cc(\mathfrak{g})=
\sum_{i\ge 1}(\dim (C^i\mathfrak{g}/C^{i+1}\mathfrak{g})-1)$.
Algebras of coclass $1$ are also called algebras of maximal class
or infinite dimensional filiform Lie algebras.
\end{definition}

\begin{example}
Let us define the algebra $L_k$ as the infinite-dimensional Lie algebra of
polynomial vector fields
on the real line ${\mathbb R}^1$ with a zero in $x=0$ of
order not less then $k+1$.

The algebra   $L_k$ can be defined by its infinite basis and  commutating relations
$$e_i=x^{i+1}\frac{d}{dx}, \; i \in {\mathbb N},\; i \ge k, \quad \quad
[e_i,e_j]= (j-i)e_{i{+}j}, \; \forall \;i,j \in {\mathbb N}.$$

$L_1$ is a residual nilpotent Lie algebra  of maximal class
generated by $e_1$ and $e_2$.
\end{example}

We recall that ${\mathbb Z}$-graded Lie algebra $W$, defined by
the basis $e_i,\; i \in {\mathbb Z}$, and relations
$$[e_i.e_j]=(j-i)e_{i+j} \quad \forall i,j \in {\mathbb Z},$$
is called the Witt algebra \cite{Fu}. Hence, the algebra $L_1$ is
the positive part $W_+=\oplus_{i>0}(W)_i$ of the Witt algebra.

One can consider another examples of algebras of maximal class
that are ${\mathbb N}$-graded Lie algebras.
\begin{example}
The Lie algebra $\mathfrak{m}_0$ is defined by its infinite basis
$e_1, e_2, \dots, e_n, \dots $ with commutator relations:
$$ \label{m_0}
[e_1,e_i]=e_{i+1}, \; \forall \; i \ge 2.$$
\end{example}

\begin{example}
The Lie algebra $\mathfrak{m}_2$ is defined by its infinite basis
$e_1, e_2, \dots, e_n, \dots $ and commutator relations:
$$
[e_1, e_i ]=e_{i+1}, \quad \forall \; i \ge 2; \quad \quad [e_2,
e_j ]=e_{j+2}, \quad \forall \; j \ge 3.
$$
\end{example}

\begin{theorem}[\cite{Fial},\cite{ShZ}]
Let $\mathfrak{g}=\bigoplus_{i=1}^{\infty}\mathfrak{g}_{i}$ be a
${\mathbb N}$-graded Lie algebra of maximal class. Then
$\mathfrak{g}$ is isomorphic to one (and only one) Lie algebra
from the three given ones:
$$\mathfrak{m}_0, \; \mathfrak{m}_2, \; L_1.$$
\end{theorem}

\section{Lie algebra cohomology}
\label{cohomology}
Let $\mathfrak{g}$ be a Lie algebra over ${\mathbb K}$ and
$\rho: \mathfrak{g} \to \mathfrak{gl}(V)$ its linear representation
(or in other words $V$ is a $\mathfrak{g}$-module).
We denote by $C^q(\mathfrak{g},V)$
the space of $q$-linear skew-symmetric mappings of $\mathfrak{g}$ into
$V$. Then one can consider an algebraic complex:

$$
\begin{CD}
V @>{d_0}>>
C^1(\mathfrak{g}, V) @>{d_1}>> C^2(\mathfrak{g}, V) @>{d_2}>>
\dots @>{d_{q{-}1}}>> C^q(\mathfrak{g}, V) @>{d_q}>> \dots
\end{CD}
$$
where the differential $d_q$ is defined by:

\begin{equation}
\begin{split}
(d_q f)(X_1, \dots, X_{q{+}1})=
\sum_{i{=}1}^{q{+}1}(-1)^{i{+}1}
\rho(X_i)(f(X_1, \dots, \hat X_i, \dots, X_{q{+}1}))+\\
+ \sum_{1{\le}i{<}j{\le}q{+}1}(-1)^{i{+}j{-}1}
f([X_i,X_j],X_1, \dots, \hat X_i, \dots, \hat X_j, \dots, X_{q{+}1}).
\end{split}
\end{equation}

The cohomology of the complex $(C^*(\mathfrak{g}, V), d)$ is called
the cohomology of the Lie algebra $\mathfrak{g}$
with coefficients in the representation $\rho: \mathfrak{g} \to V$.

In this article we will consider two main examples:

1) $V= {\mathbb K}$ and $\rho: \mathfrak{g} \to {\mathbb K}$ is trivial;

2) $V= \mathfrak{g}$ and $\rho=ad: \mathfrak{g} \to \mathfrak{g}$ is
the adjoint representation of $\mathfrak{g}$.

The cohomology of $(C^*(\mathfrak{g}, {\mathbb K}), d)$ (the first example)
is called the
cohomology with trivial coefficients of the Lie algebra
$\mathfrak{g}$ and is denoted by $H^*(\mathfrak{g})$.
Also we fixe
the notation $H^*(\mathfrak{g},\mathfrak{g})$ for the cohomology
of $\mathfrak{g}$ with coefficients in the adjoint representation.

Let $\mathfrak{g}=\oplus_{\alpha}\mathfrak{g}_{\alpha}$ be a
${\mathbb Z}$-graded Lie algebra
and $V=\oplus_{\beta} V_{\beta}$ is a ${\mathbb Z}$-graded
$\mathfrak{g}$-module, i.e.,
$\mathfrak{g}_{\alpha}V_{\beta} \subset V_{\alpha {+} \beta}.$
Then the complex $(C^*(\mathfrak{g}, V), d)$
can be equipped with the ${\mathbb Z}$-grading
$C^q(\mathfrak{g},V) =
\bigoplus_{\mu} C^q_{(\mu)}(\mathfrak{g},V)$, where
a $V$-valued $q$-form $c$ belongs to
$C^q_{(\mu)}(\mathfrak{g},V)$
iff for
$X_1 \in \mathfrak{g}_{\alpha_1},
\dots,  X_q \in \mathfrak{g}_{\alpha_q}$ we have
$$c(X_1,\dots,X_q) \in
V_{\alpha_1{+}\alpha_2{+}\dots{+}\alpha_q{+}\mu}.$$

This grading is compatible with the differential  $d$ and
hence we have ${\mathbb Z}$-grading in cohomology:
$$
H^{q} (\mathfrak{g},V)= \bigoplus_{\mu \in {\mathbb Z}}
H^{q}_{(\mu)} (\mathfrak{g},V).
$$

\begin{remark}
The trivial $\mathfrak{g}$-module ${\mathbb K}$ has only one non-trivial
homogeneous component ${\mathbb K}={\mathbb K}_0$.
\end{remark}

\begin{example}
Let $\mathfrak{g}$ be an infinite dimensional Lie algebra with the
infinite basis $e_1, e_2, \dots, e_n, \dots$ and commutating
relations
$$[e_i,e_j]= c_{ij}e_{i{+}j}.$$

Let us consider the dual basis $e^1, e^2, \dots, e^n, \dots$. One
can introduce a grading (that we will call the weight) of
$\Lambda^*(\mathfrak{g}^*)=C^*(\mathfrak{g},{\mathbb K})$:
$$\Lambda^* (\mathfrak{g}^*)=
\bigoplus_{\lambda{=}1}^{\infty} \Lambda^*_{(\lambda)}
(\mathfrak{g}^*),$$ where a subspace $\Lambda^{q}_{(\lambda)}
(\mathfrak{g}^*)$ is spanned by $q$-forms $\{ e^{i_1} {\wedge}
\dots {\wedge} e^{i_q}, \; i_1{+}\dots{+}i_q {=} \lambda \}$. For
instance a monomial $e^{i_1} \wedge \dots \wedge e^{i_q}$ has the
degree $q$ and the weight $\lambda=i_1{+}\dots{+}i_q$.

The complex $(C^*(\mathfrak{g}, \mathfrak{g}), d)$ is
${\mathbb Z}$-graded:
$$C^*(\mathfrak{g}, \mathfrak{g})=
\bigoplus_{\mu \in {\mathbb Z}}C^*_{(\mu)}(\mathfrak{g}, \mathfrak{g}),$$
where $C^q_{(\mu)}(\mathfrak{g}, \mathfrak{g})$ is spanned
by monomials
$\{ e_l \otimes e^{i_1} {\wedge} \dots {\wedge} e^{i_q}, \;
i_1{+}\dots{+}i_q{+}\mu =l \}$.
\end{example}

Let $\mathfrak{g}= \oplus_{\alpha >0} \mathfrak{g}_{\alpha}$ be a
${\mathbb N}$-graded Lie algebra. One can define a decreasing
filtration ${\mathcal F}$ of the adjoint cochain complex
$(C^*(\mathfrak{g},\mathfrak{g}),d)$ of ${\mathfrak{g}}$:
$$
{\mathcal F}^0 C^*(\mathfrak{g},\mathfrak{g})\supset \dots \supset
{\mathcal F}^q C^*(\mathfrak{g},\mathfrak{g}) \supset {\mathcal
F}^{q{+}1} C^*(\mathfrak{g},\mathfrak{g}) \supset \dots
$$
where the subspace ${\mathcal F}^q
C^{p{+}q}(\mathfrak{g},\mathfrak{g})$ is spanned by $p{+}q$-forms
$c$ in $C^{p{+}q}(\mathfrak{g},\mathfrak{g})$ such that
$$
c(X_1,\dots,X_{p{+}q}) \in \bigoplus_{\alpha=q}^{+\infty}
{\mathfrak g}_{\alpha}, \; \forall X_1,\dots,X_{p{+}q} \in
\mathfrak{g}.
$$

The filtration ${\mathcal F}$ is compatible with $d$.

Let us consider the corresponding spectral sequence $E_r^{p,q}$:
\begin{proposition}
\label{spectralseq} $E_1^{p,q}={\mathfrak g}_q \otimes
H^{p{+}q}({\mathfrak g})$.
\end{proposition}

We have the following natural isomorphisms:
\begin{equation}
\begin{split}
C^{p{+}q}({\mathfrak g},{\mathfrak g})
= {\mathfrak g} \otimes \Lambda^{p{+}q}({\mathfrak g}^*)\\
E_0^{p,q}={\mathcal F}^q C^{p{+}q}(\mathfrak{g},\mathfrak{g})/
{\mathcal F}^{q{+}1} C^{p{+}q}(\mathfrak{g},\mathfrak{g})
={\mathfrak g}_q \otimes \Lambda^{p{+}q}({\mathfrak g}^*).
\end{split}
\end{equation}

Now the proof follows from the formula for the differential
$d_0^{p,q}: E_0^{p,q} \to E_0^{p{+}1,q}$:
$$d_0(X \otimes f)=X\otimes df,$$
where $X \in \mathfrak{g}, f \in \Lambda^{p{+}q}({\mathfrak g}^*)$
and $df$ is the standart differential of the cochain complex of
$\mathfrak{g}$ with trivial coefficients.
\begin{remark}
The spectral sequence considered above was used by Feigin and
Fukhs \cite{FeFu} in their computations of $H^*(L_1,L_1)$.
\end{remark}

\section{Scalar cohomology of $\mathfrak{m}_0$}
The cohomology algebra $H^*(\mathfrak{m}_0)$ was calculated in
\cite{FialMill}. We will briefly recall some results from this
article.

It were introduced two operators:

1) $D_1=ad^*e_1: \Lambda^*(e_2, e_3, \dots) \to \Lambda^*(e_2,
e_3, \dots)$,
\begin{equation}
\begin{split}
D_1(e^2)=0, \; D_1(e^i)= e^{i-1}, \; \forall i\ge 3,\\
D_1(\xi\wedge \eta)=D_1(\xi)\wedge \eta +\xi\wedge D_1(\eta), \;
\: \forall \xi, \eta \in \Lambda^*(e_2, e_3, \dots).
\end{split}
\end{equation}

2) and its right inverse $D_{-1}: \Lambda^*(e^2,e^3,\dots) \to
\Lambda^*(e^2,e^3,\dots)$,
\begin{equation}
\begin{split}
\label{D_{-1}}
e^i=e^{i+1}, \; D_{-1}(\xi{\wedge} e^i)= \sum_{l\ge
0}(-1)^l D_{1}^{l}(\xi){\wedge} e^{i+1+l},
\end{split}
\end{equation}
where $i\ge 2$ and $\xi$ is an arbitrary form in
$\Lambda^*(e^2,\dots,e^{i-1})$. The sum in the definition
(\ref{D_{-1}}) of $D_{-1}$ is always finite because $D_1^l$
decreases the second grading by $l$. For instance,
$$D_{-1}(e^i\wedge e^k)=\sum_{l=0}^{i-2} ({-}1)^l e^{i-l}{\wedge} e^{k+l+1}.$$

\begin{proposition}
The operators $D_1$ and $D_{-1}$ have the following properties:
$$
\label{D_1,D_{-1}} d\xi=e^1\wedge D_1\xi, \; e^1\wedge \xi=d
D_{-1}\xi, \; D_1D_{-1}\xi=\xi, \quad \xi \in
\Lambda^*(e^2,e^3,\dots).
$$
\end{proposition}

\begin{theorem}[\cite{FialMill}]
The infinite dimensional bigraded cohomology
$H^*(\mathfrak{m}_0)=\oplus_{k,q} H^q_k(\mathfrak{m}_0)$ is
spanned by the cohomology classes of $e^1$, $e^2$ and of the
following homogeneous cocycles:
\begin{equation}
\label{cocycles} \omega_I=\omega(e^I)=\omega(e^{i_1}{\wedge}\dots
{\wedge} e^{i_q} {\wedge} e^{i_q{+}1})= \sum\limits_{l\ge 0}(-1)^l
D_1^l(e^{i_1}\wedge \dots \wedge e^{i_q})\wedge e^{i_q+1+l},
\end{equation}
where $q \ge 1, \; I=(i_1,\dots,i_{q-1},i_q,i_q+1), 2\le i_1
{<}i_2{<}{\dots} {<}i_q$. The multiplicative structure is defined
by
\begin{multline}
\label{multiplicat} [e^1] {\wedge}
\omega(\xi{\wedge}e^i{\wedge}e^{i+1})=0,\; [e^2] {\wedge}
\omega(\xi{\wedge}e^i{\wedge}e^{i+1})
=\omega(e^2{\wedge}\xi{\wedge}e^i{\wedge}e^{i+1}),\\
\omega(\xi{\wedge}e^i{\wedge}e^{i+1}) {\wedge}
\omega(\eta{\wedge}e^j{\wedge}e^{j+1})
=\sum_{l=0}^{j-i-2}(-1)^l\omega(D_1^l(\xi{\wedge}e^i){\wedge}e^{i+1+l}
{\wedge}\eta{\wedge}e^j{\wedge}e^{j+1})+\\
+(-1)^{i{-}j{+}deg\eta}\sum_{s \ge 1}
\omega((ad^*e_1)^{i{-}j{-}1{+}s}(\xi{\wedge}e^i){\wedge}
D_1^s(\eta{\wedge}e^j){\wedge}e^{j{+}s} {\wedge}e^{j{+}s{+}1}),
\end{multline}
where $i < j$, $\xi$ and $\eta$ are arbitrary homogeneous forms in
$\Lambda^*(e^2,\dots,e^{i{-}1})$ and
$\Lambda^*(e^2,\dots,e^{j{-}1})$, respectively.
\end{theorem}

Formula (\ref{cocycles}) determines a homogeneous closed
$(q{+}1)$-form of the second grading
$i_1{+}{\dots}{+}i_{q{-}1}{+}2i_{q}{+}1$. It has only one monomial
in its decomposition of the form $\xi \wedge e^i \wedge e^{i+1}$
and it is $e^{i_1}{\wedge}\dots {\wedge} e^{i_q} {\wedge}
e^{i_q{+}1}$.

The whole number of linearly indepenedent $q$-cocycles of the
second grading $k{+}\frac{q(q{+}1)}{2}$ is equal to
$${\rm dim} H_{k{+}\frac{q(q{+}1)}{2}}^q(\mathfrak{m}_0)=
P_q(k)-P_q(k{-}1),$$ where $P_q(k)$ denotes the number of
(unordered) partitions of a positive integer $k$ into $q$ parts.

\begin{example}
\label{3cocycle}
\begin{multline*}
\omega_{(5,6,7)}=\omega(e^5{\wedge} e^6 {\wedge} e^7)= e^5{\wedge}
e^6 {\wedge}e^7-e^4{\wedge} e^6 {\wedge} e^8+(e^3{\wedge} e^6+
e^4 {\wedge} e^5){\wedge} e^9-\\
-(e^2{\wedge}e^6+2e^3 {\wedge}e^5){\wedge}e^{10} +(3e^2{\wedge}
e^5+2e^3{\wedge} e^4){\wedge} e^{11}- 5e^2{\wedge} e^4{\wedge}
e^{12}+5e^2{\wedge} e^3{\wedge} e^{13}.
\end{multline*}
\end{example}

\begin{proposition}
\label{umnozh} It follows from (\ref{multiplicat}) that
$$
\omega(e^2\wedge e^3\wedge\dots\wedge e^i\wedge
e^{i+1})\wedge\omega_I=\omega(e^2\wedge e^3\wedge\dots\wedge
e^i\wedge e^{i+1}\wedge e^I).
$$
\end{proposition}

\section{The spectral sequence and $H^*({\mathfrak{m}_0}, {\mathfrak{m}_0}$)}
\label{Computations} In this section we compute
$H^*({\mathfrak{m}_0}, {\mathfrak{m}_0})$.

\begin{theorem}
\label{main} The bigraded cohomology
$H^*(\mathfrak{m}_0,\mathfrak{m}_0)=\oplus_{p,k}
H^p_k(\mathfrak{m}_0,\mathfrak{m}_0)$ is an infinite dimensional
linear space of formal series $\sum_{J,s} \alpha_{J,s}\Psi_{J,s}$
of the following infinite system $\{\Psi_{J,s}\}$ of homogeneous
cocycles:
\begin{equation}
\label{formulirovka}
\begin{split}
\Psi_{1,1}=e_1\otimes e^1+\sum_{j=3}^{\infty}(j{-}2)e_j\otimes
e^j,\; \Psi_{1,2}=e_2\otimes e^1, \;
\Psi_{2,l+2}=\sum_{j=2}^{\infty}e_{l{+}j}\otimes e^j, \; l \ge 0,
l\ne 1,
\\
\Psi_{I,r}=\sum_{j=0}^{\infty}e_{r{+}j} \otimes D_{-1}^j\omega_I,
\;r {\ge} 2, \; I{=}(i_1,\dots,i_q,i_q+1),\quad q{\ge}1, \\
2\le i_1{<}{\dots} {<}i_q,\;
i_{r-2}>r-1,\; {\rm if} \quad 3 \le r \le q+1, r= q+3.\\
\end{split}
\end{equation}
where $\omega_I$ stands for a basic scalar cocycle defined by
(\ref{cocycles}) and $D_{-1}$ is the operator defined by
(\ref{D_{-1}}).

The homogeneous cocycles $\Psi_{J,s}$ have the following gradings:
\begin{equation}
\label{main_formula}
\begin{split}
\Psi_{1,1}, \Psi_{2,2} \in H^1_0(\mathfrak{m}_0,\mathfrak{m}_0),
\; \Psi_{2,1} \in H^1_1(\mathfrak{m}_0,\mathfrak{m}_0), \;
\Psi_{2,l+2} \in H^1_l(\mathfrak{m}_0,\mathfrak{m}_0),
\\ \Psi_{I,r} \in H^{q+1}_k(\mathfrak{m}_0,\mathfrak{m}_0), \; I{=}(i_1,\dots,i_q,i_q+1), \;k
=r-\left(i_1{+}i_2{+}\dots{+}i_{q-1}{+}2i_q{+}1\right).
\end{split}
\end{equation}
The cocycle $\Psi_{I,r}$ with $I{=}(i_1,i_2,\dots,i_q,i_q+1)$ is
uniquely determined by the following condition:
\begin{equation}
\label{property_Psi}
\begin{split}
\Psi_{I,r}\left(e_{i_1},e_{i_2},\dots,e_{i_q},e_{i_q+1}\right)=e_r,
\quad 2\le i_1{<}{\dots} {<}i_q,\\
\Psi_{I,r}\left(e_{j_1},e_{j_2},\dots,e_{j_q},e_{j_q+1}\right)=0,
\quad 2\le j_1{<}{\dots} {<}j_q, \; (j_1,\dots,j_q,j_{q}+1) \ne I.
\end{split}
\end{equation}

\end{theorem}

\begin{proof}
We consider the spectral sequence $E_r^{p,q}$ from the Proposition
\ref{spectralseq}. Namely it follows  that
$$E_1^{p,q}=({\mathfrak{m}_0})_q \otimes H^{p{+}q}({\mathfrak{m}_0}).$$
\begin{proposition}
1) The first differential $d_1:E_1^{p,q}\to E_1^{p,q+1}$ is non
trivial  only in the cases:
$$
d_1(e_q)=-e_{q+1}\otimes e^1, \; q \ge 2.
$$

2) For the second differential $d_2:E_2^{p,q}\to E_2^{p-1,q+2}$ we
have the following property:
$$
d_2(e_1)=e_3 \otimes e^2, \quad d_2(e_1 \otimes \omega)=e_3\otimes
e^2 {\wedge}\omega, \quad d_2(e_q \otimes \omega)=0,  \; q \ge 2.
$$
\end{proposition}
\begin{proof}
First of all by the definiton of
$d:C^0(\mathfrak{g},\mathfrak{g}){=}\mathfrak{g} \to
C^1(\mathfrak{g},\mathfrak{g}){=} \mathfrak{g} {\otimes}
{\mathfrak g}^*$ we have
$$dX(Y)=[X,Y], \; X,Y \in \mathfrak{g}.$$
Hence for an arbitrary closed $\omega$ we have
\begin{equation}
\begin{split}
d e_1 = \sum_{j=2}^{\infty} e_{j+1} \otimes e^j, \quad d e_j = -
e_{j{+}1} \otimes e^1, \; j > 1;\\
d (e_1 \otimes \omega)=e_3\otimes e^2 {\wedge}\omega +\dots, \quad
d(e_q \otimes \omega)=e_{q+1}\otimes e^1 \wedge \omega , q
>1.\end{split}
\end{equation}

We recall that the product $e^1 \wedge \omega$ is always trivial
in cohomology $H^*(\mathfrak{m}_0,\mathfrak{m}_0)$ and we can
shift our form $e_q \otimes \omega$ by $e_{q+1} \otimes
D_{-1}\omega$ and see that
$$
d(e_q \otimes \omega + e_{q+1} \otimes D_{-1}\omega)=-e_{q+2}
\otimes e^1 \wedge D_{-1}\omega+\dots
$$
We put everywhere dots instead of terms of higher filtration.
\end{proof}

It follows from the Proposition \ref{umnozh} that the class
$e^2\wedge \omega$ is trivial if and only if $\omega = e^2 \wedge
\tilde \omega$, for some $\tilde \omega$. We came to the following
corollary:

\begin{corollary}
\label{maincorl} The following classes in $E_1$ survive to $E_3$:
\begin{equation}
\begin{split}
e_1 \otimes e^1, \quad e_1 \otimes {e^2 {\wedge} \omega},\quad e_2 \otimes e^1, \quad e_q \otimes e^2, \; q \ge 2\\
e_{r} \otimes \omega_{I}, \; r \ge 2, (r,I) \ne
(3,(2,i_2,\dots,i_q,i_q+1)).
\end{split}
\end{equation}
\end{corollary}

\begin{proposition}
The differential $d_s, s \ge 3$, can be non trivial only in the
following case:
$$
d_s(e_1 \otimes e^2\wedge e^3 \wedge \dots \wedge e^{s-1} \wedge
\omega)=(-1)^s e_{s+1} \otimes e^2\wedge e^3\wedge \dots \wedge
e^{s-1} \wedge e^{s}\wedge \omega.
$$
\end{proposition}
\begin{proof}
A cocycle $\Psi_{I,r}$ (\ref{formulirovka}) represents $e_r
\otimes \omega_{I}$ and  obviously $d_s(\Psi_{I,r})=0, s \ge 2$.
Class $e_1\otimes \omega$ survives to $E_{\infty}$ if and only if
$\omega$ is divisible by $e^2\wedge e^3 \wedge \dots \wedge e^s$
for arbitrary $s \ge 2$. Hence no one class of the form $e_1
\otimes \omega$ can survive to $E_{\infty}$.

From the another hand  we have to take the quotient of $E_s$ over
the subspace of classes of the form $e_{s+1}\otimes e^2\wedge e^3
\wedge \dots \wedge e^s \wedge \tilde \omega, s \ge 2$. It means
that we have to remove from our final list the cocycles
$\Psi_{I,s+1}$ with $I$ such as $i_1=2, \dots, i_{s-1}=s$ if $s-1
< q$ as well as the cocycles $\Psi_{(2,3,\dots,s),s+1}$.
\end{proof}

Now we leave to the reader to prove the property
(\ref{property_Psi}) of the basic cocycles $\Psi_{I,r}$.
\end{proof}

\begin{remark}
It follows from the theorem \ref{main} that zero-cohomology of
$\mathfrak{m}_0$ is trivial:
$$H^0(\mathfrak{m}_0,\mathfrak{m}_0)=0.$$

One-dimensional cohomology $H^1(\mathfrak{m}_0,\mathfrak{m}_0)$ is
infinite dimensional, but its homogeneous components are finite
dimensional ones:
\begin{equation}
\begin{split}
 H^1_k(\mathfrak{m}_0,\mathfrak{m}_0)=0, k \le -1, \;
H^1_0(\mathfrak{m}_0,\mathfrak{m}_0)=\langle \Psi_{1,1},
\Psi_{2,2}\rangle, \\ H^1_1(\mathfrak{m}_0,\mathfrak{m}_0)=\langle
\Psi_{1,2} \rangle, \;
H^1_k(\mathfrak{m}_0,\mathfrak{m}_0)=\langle \Psi_{2,k+2}\rangle,
\; k \ge 2.
\end{split}
\end{equation}
The structure of $H^1(\mathfrak{m}_0,\mathfrak{m}_0)$  was the
subject in \cite{FialWag}, but it was found earlier in \cite{GKh1,
Kh}  (for instance our cocycles $\Psi_{1,1}, \Psi_{2,2}$ coinside
with $\omega_1, \omega_2$ in \cite{FialWag} and they are also
equal to $d_1$ and $d_2-2d_1$ in \cite{GKh1, Kh}). Let us recall
that an arbitrary positively graded Lie algebra $\mathfrak
g=\oplus_i {\mathfrak g}_i$ has a derivation $\tau$ defined by
means of its graded structure
$$\tau(X)=iX, \; X \in {\mathfrak
g}_i.$$ In our case $\tau = \Psi_{1,1}+2\Psi_{2,2}$. Another
obvious fact is that the one-dimensional cohomology $H^1(\mathfrak
g,\mathfrak g)$ endowed with the Nijehhuis-Richardson bracket is
isomorphic to the algebra of outer derivations of $\mathfrak g$.
\end{remark}

\begin{example}
We have defined the two-dimensional cocycle $\Psi_{i,i+1,r}$ by
(\ref{main_formula}) as
$$\Psi_{i,i+1,r}=\sum_{l=0}^{\infty}e_{r{+}l} \otimes
D_{-1}^l\omega_{i,i+1}.$$ An elementary exercise on the properties
of the operator $D_{-1}$ will be to verify that
$$
D_{-1}^l\omega_{i,i+1}=\sum_{s=0}^{i-2}(-1)^s\binom
{l+s}{s}e^{i-s}\wedge e^{i+1+s+l}.
$$
Hence we have the following formula
\begin{equation}
\Psi_{i,i+1,r}=\sum_{l=0}^{\infty}\sum_{s=0}^{i-2}(-1)^s\binom{l+s}{s}e_{r{+}l}
\otimes e^{i-s}\wedge e^{i+1+s+l}.
\end{equation}
It means that
$$
\Psi_{i,i+1,r}\left(e_{k},e_{m}\right)=(-1)^{i-k}\binom{m-i-1}{i-k}e_{m+k-2i-1},
\; 2 \le k \le i < m.
$$
Hence the cocycles $\Psi_{i,i+1,r}$ coinside with the basic
cocycles $\Psi_{i,r}$ considered by Khakimdjanov in \cite{Kh}
where he used Vergne's general algorithm  \cite{V}.

It is evident that a homogeneous subspace
$H^2_k(\mathfrak{m}_0,\mathfrak{m}_0)$ is infinite dimensional as
it was remarked in \cite{FialWag}.
\end{example}

\begin{example}
The infinite dimensional $H^3({\mathfrak m}_0,{\mathfrak m}_0)$ is
the space of formal series of basic cocycles
$$
\Psi_{(j,i,i+1),r}=e_r \otimes
\omega(e^j{\wedge}e^i{\wedge}e^{i+1})+e_{r+1}\otimes
D_{-1}\omega(e^j{\wedge}e^i{\wedge}e^{i+1})+\dots, r \ge 2, \; 2
\le j < i,
$$
where the cocycles $\left\{\Psi_{(2,i,i+1),3}, i>2\right\}$, and
$\Psi_{(2,3,4),5}$ have been removed from the list of basic
cocycles according to the rule from (\ref{formulirovka}).
\end{example}

\section{The scalar cohomology $H^*(\mathfrak{m}_2)$}

The operator $D_2=ad^*e_2: \Lambda^*(\mathfrak{b}) \to
\Lambda^*(\mathfrak{b})$ inducing ${\mathcal D}_2:
H^*(\mathfrak{b}) \to H^*(\mathfrak{b})$ can be defined by
\begin{equation}
\begin{split}
D_2(e^1)=D_2(e^4)=0,\; D_2(e^3)=e^1,\; D_2(e^i)= e^{i-2},
\: i\ge 5,\\
D_2(\xi\wedge \eta)=D_2(\xi)\wedge \eta +\xi\wedge D_2(\eta), \;
\forall \xi, \eta \in \Lambda^*(\mathfrak{b}).
\end{split}
\end{equation}

It is immediate that
\begin{equation}
\label{2-surj}
\begin{split}
{\mathcal D}_2(e^3)=e^1&, \quad {\mathcal D}_2(e^1)=0,\\
{\mathcal D}_2(\omega_{\mathfrak{b}}(e^3{\wedge} e^4))=0,& \;
{\mathcal D}_2(\omega_{\mathfrak{b}}(e^k{\wedge} e^{k+1}))=
-2\omega_{\mathfrak{b}}(e^{k-1}{\wedge} e^k), \; k \ge 4.
\end{split}
\end{equation}
\begin{proposition} Let $3\le i_1<\dots<i_{p-1}<i$ and
$\xi=e^{i_1}{\wedge}\dots{\wedge}e^{i_{p-1}}$, then
$$
\label{mathcal D_2} {\mathcal
D}_2(\omega_{\mathfrak{b}}(\xi{\wedge}e^{i}{\wedge}e^{i+1}))=
\omega_{\mathfrak{b}}((D_2+D_1^2)(\xi){\wedge}e^{i}{\wedge}e^{i+1})
-2\omega_{\mathfrak{b}}(\xi{\wedge}e^{i-1}{\wedge}e^{i}).$$
\end{proposition}

\begin{theorem}
\label{main_H_m_2} The bigraded cohomology algebra
$H^*(\mathfrak{m}_2)=\oplus_{q,k} H^q_k(\mathfrak{m}_2)$ is
spanned by cohomology classes of the following homogeneous
cocycles:
\begin{multline}
\label{w_cocycles}
e^1, \; e^2, \;  e^2 \wedge e^3, \; e^3\wedge e^4-e^2\wedge e^5, \\
w_{i_1{,} {\dots}{,} i_q{,} i_q{+}1{,} i_q{+}2}=
\sum_{l{\ge}1}\frac{1}{2^l}\omega\left({(}D_2{+}D_1^2{)}^l
{(}e^{i_1}{\wedge}{\dots} {\wedge} e^{i_q}{)} {\wedge}
e^{i_q{+}1{+}l}{\wedge} e^{i_q{+}2{+}l}\right),
\end{multline}
where $1\le q, \; 3\le i_1 <i_2<\dots <i_q$, in particular for $q
\ge 3$,
$$
{\rm dim} H_{k{+}\frac{q(q{+}1)}{2}}^q(\mathfrak{m}_2)=
P_q(k)-P_q(k{-}1)-P_q(k{-}2)+ P_q(k{-}3).
$$
\end{theorem}
\begin{example}
\begin{multline*}
w_{5,6,7}=\omega(e^5{\wedge} e^6 {\wedge} e^7) +
\omega(e^3{\wedge} e^7{\wedge}e^8)
=e^5{\wedge} e^6 {\wedge}e^7+(e^3{\wedge} e^7{-}e^4{\wedge} e^6){\wedge} e^8+\\
{+}(e^4 {\wedge} e^5{-}e^2{\wedge} e^7){\wedge} e^9+
(e^2{\wedge}e^6{-}e^3 {\wedge}e^5){\wedge}e^{10} +e^3{\wedge}
e^4{\wedge} e^{11}- e^2{\wedge} e^4{\wedge} e^{12}+e^2{\wedge}
e^3{\wedge} e^{13}.
\end{multline*}
\end{example}

1) The space $H^2(\mathfrak{m}_2)$ is two-dimensional and it is
spanned by the cohomology classes represented by cocycles
$e^2\wedge e^3$ and $e^3\wedge e^4-e^2\wedge e^5$ of second
gradings $5$ and $7$ respectively;

2) $H^3(\mathfrak{m}_2)$ is infinite dimensional and it is spanned
by
$$w_{k,k{+}1,k{+}2}=
\sum_{l\ge 0}\omega\left(e^{k-2l}{\wedge} e^{k+1+l}{\wedge}
e^{k+2+l}\right), \; k\ge 3.$$

\section{Adjoint cohomology $H^*({\mathfrak{m}_2}, {\mathfrak{m}_2}$)}
\begin{theorem}
\label{main} The bigraded cohomology
$H^*(\mathfrak{m}_2,\mathfrak{m}_2)=\oplus_{p,k}
H^p_k(\mathfrak{m}_2,\mathfrak{m}_2)$ is an infinite dimensional
linear space of formal series $\sum_{J,s} \alpha_{J,s}\Phi_{J,s}$
of the following infinite system $\{\Phi_{J,s}\}$ of homogeneous
cocycles:
\begin{equation}
\label{formulirovka}
\begin{split}
\Phi_{1,1}=\sum_{j=1}^{\infty}j e_j\otimes e^j,\;
\Phi_{2,l+2}=\sum_{j=2}^{\infty}e_{l{+}j}\otimes e^j, \; l \ge 3,
\quad \Phi_{2,3,m}=\sum_{j=0}^{\infty}e_{m{+}j}\otimes e^2\wedge
e^{3+j},\;m=3, m\ge 7,\\
\Phi_{2,3,1}=e_1\otimes e^2{\wedge}
e^{3}+\frac{1}{2}\sum_{j=0}^{\infty}e_{5{+}j}\otimes
\left(e^4{\wedge} e^{5{+}j}{-}(j{+}1)e^3{\wedge}
e^{6{+}j}{+}\frac{(j{+}2)(j{+}1)}{2}e^2{\wedge} e^{7{+}j}\right),\\
\Phi_{2,3,2}=\sum_{i=0}^{\infty}e_{2{+}i}\otimes e^2{\wedge}
e^{3+i}+\frac{1}{2}\sum_{j=0}^{\infty}e_{6{+}j}\otimes
\left(e^4{\wedge} e^{5{+}j}{-}(j{+}1)e^3{\wedge}
e^{6{+}j}{+}\frac{(j{+}2)(j{+}1)}{2}e^2{\wedge} e^{7{+}j}\right),\\
\Phi_{3,4,l}=\sum_{i=0}^{\infty}e_{l{+}i}\otimes \left(e^3\wedge
e^{4+i}-(i+1) e^2\wedge e^{5+i}\right), \; l \ge 3,\\
 \Phi_{I,r}=\sum_{j=0}^{\infty}e_{r{+}j}
\otimes \tilde D_{-1}^jw_I, \;r {\ge} 3, \;
I{=}(i_1,\dots,i_q,i_q+1,i_q+2),\quad q{\ge}1, \quad 3\le
i_1{<}{\dots} {<}i_q,\\
{\rm if}\;\; r=4, q\ge 2, \;\;{\rm then}\;\; i_1> 3,\\ {\rm
if}\;\; 5\le r \le q+3,\; r= q+6,\;\;{\rm then}\; i_{r-4}>r-1,
\;\;{\rm or}\;\;
i_{r-4}=r-1,\; i_1>3,\\
{\rm if}\;\; r=q+5, \;\;{\rm then}\;\; i_{r-3} > r-1.
\end{split}
\end{equation}
with $w_I$ stands for a basic scalar cocycle defined by
(\ref{w_cocycles}), $\tilde D_{-1}^j w_{I}$ defined by induction
by $d\tilde D_{-1}^j w_{I}= e^1\wedge \tilde
D_{-1}^{j-1}w_{I}+e^2\wedge D_{-1}^{j-2}w_{I}$ and $d\tilde D_{-1}
w_{I}= e^1\wedge \tilde w_{I}$.

The homogeneous cocycles $\Phi_{J,s}$ have the following gradings:
\begin{equation}
\label{main_formula}
\begin{split}
\Phi_{1,1} \in H^1_0(\mathfrak{m}_2,\mathfrak{m}_2), \;
\Phi_{2,l+2} \in H^1_l(\mathfrak{m}_2,\mathfrak{m}_2),
\Phi_{2,3,l} \in H^2_{l-5}(\mathfrak{m}_2,\mathfrak{m}_2),
\Phi_{3,4,l} \in H^2_{l-7}(\mathfrak{m}_2,\mathfrak{m}_2),
\\ \Phi_{I,r} \in H^{q+2}_k(\mathfrak{m}_0,\mathfrak{m}_0), \; I{=}(i_1,\dots,i_q,i_q+1,i_q+2), \;k
=r-\left(i_1{+}i_2{+}\dots{+}i_{q-1}{+}3i_q{+}3\right).
\end{split}
\end{equation}
The cocycle $\Phi_{I,r}$ with
$I{=}(i_1,i_2,\dots,i_q,i_q+1,i_q+2)$ is uniquely determined by
the following condition:
\begin{equation}
\label{property}
\begin{split}
\Phi_{I,r}\left(e_{i_1},e_{i_2},\dots,e_{i_q},e_{i_q+1},e_{i_q+2}\right)=e_r,
\quad 2\le i_1{<}{\dots} {<}i_q,\\
\Phi_{I,r}\left(e_{j_1},e_{j_2},\dots,e_{j_q},e_{j_q+1},e_{j_q+2}\right)=0,
\quad 2\le j_1{<}{\dots} {<}j_q, \; (j_1,\dots,j_q,j_{q}+1,j_q+2)
\ne I.
\end{split}
\end{equation}
\end{theorem}

\begin{proof}
We consider the spectral sequence considered above. It follows
from the Proposition \ref{spectralseq} that
$$E_1^{p,q}=({\mathfrak{m}_2})_q \otimes H^{p{+}q}({\mathfrak{m}_2}).$$
\begin{proposition}
1) The first differential $d_1:E_1^{p,q}\to E_1^{p,q+1}$ is non
trivial  only in the cases:
$$
d_1(e_q)=-e_{q+1}\otimes e^1, \; q \ge 2.
$$

2) The second differential $d_2:E_2^{p,q}\to E_2^{p-1,q+2}$ has
the only one non trivial value:
$$
d_2(e_1)=e_3 \otimes e^2.
$$
\end{proposition}
\begin{proof}
We start with
\begin{equation}
\begin{split}
d e_1 = \sum_{j=2}^{\infty} e_{j+1} \otimes e^j, \quad d e_2 =-e_3
\otimes e^1+ \sum_{j=2}^{\infty} e_{j+2} \otimes e^j, \quad d e_j
= -
e_{j{+}1} \otimes e^1-e_{j{+}2} \otimes e^2, \; j \ge 3;\\
d (e_1 \otimes \omega)=e_3\otimes e^2 {\wedge}\omega +\dots, \quad
d(e_q \otimes \omega)=e_{q+1}\otimes e^1 \wedge \omega +\dots \
\end{split}
\end{equation}

We recall that the products $e^1 \wedge \omega$ and $e^2 \wedge
\omega$ and are always trivial in cohomology $H^*(\mathfrak{m}_2)$
for an arbitrary closed form $\omega$. Also one can remark that a
form $e^1 \wedge \xi_1 + e^2 \wedge \xi_2$ is closed if and anly
if it is exact. Hence we can shift our form $e_q \otimes \omega$
by $e_{q+1} \otimes {\tilde D}_{-1}\omega$, where $d{\tilde
D}_{-1}\omega=e^1 \wedge \omega$ and see that
$$
d(e_q \otimes \omega + e_{q+1} \otimes \tilde
D_{-1}\omega)=-e_{q+2} \otimes \left(e^1 \wedge \tilde
D_{-1}\omega+e^2\wedge \omega \right)+\dots
$$
We put everywhere dots instead of terms of higher filtration.
\end{proof}

It follows from the Proposition \ref{umnozh} that the class
$e^2\wedge \omega$ is trivial if and only if $\omega = e^2 \wedge
\tilde \omega$, for some $\tilde \omega$. We came to the following
corollary:

\begin{corollary}
\label{maincorl} The following classes in $E_1$ survive to $E_3$:
\begin{equation}
\begin{split}
e_1 \otimes e^1,  \quad e_2 \otimes e^1, \quad e_1 \otimes e^2,\quad e_2 \otimes e^2,\quad e_q \otimes e^2,\; q \ge 4,\\
e_l\otimes e^2\wedge e^3, e_l\otimes \left(e^3\wedge e^4-e^2\wedge
e^5\right), l \ge 1, \quad e_{r} \otimes w_{I}, \; r \ge 1.
\end{split}
\end{equation}
\end{corollary}

\begin{proposition}
The differentials $d_s, 3 \le s \le 4$ are non trivial in the
following cases:
\begin{multline}
\label{d_4}
 d_3(e_1 \otimes e^2)=-e_4\otimes e^2\wedge
e^3,\;d_3(e_1\otimes w_{i_1,i_2,i_3,\dots,i_q{+}2})=e_4\otimes
w_{3,i_1,i_2,i_3,\dots,i_q{+}2},\\
d_3(e_2 \otimes e^2)=-e_5\otimes e^2\wedge e^3,\;
d_3(e_2\otimes w_{i_1,i_2,i_3,\dots,i_q{+}2})=e_5\otimes w_{3,i_1,i_2,i_3,\dots,i_q{+}2}, i_1 > 3, \\
d_4(e_2\otimes e^1)=2e_6\otimes e^2\wedge e^3,\; d_4(e_2\otimes
w_{3,i_2,i_3,\dots,i_q{+}2})= -e_6\otimes
w_{3,4,i_2,i_3,\dots,i_q{+}2}, i_2 >4.
\end{multline}
\end{proposition}

\begin{proof}
We will prove some of the formulas (\ref{d_4}), the rest of them
can be obtained analogously
\begin{multline}
 d(e_1 \otimes e^2)=e_4\otimes e^3\wedge e^2+\dots,\quad
 d(e_2 \otimes e^2+ e_3 \otimes e^3+ e_4 \otimes e^4)=e_5\otimes \left(-2e^2\wedge e^3-e^1\wedge e^4\right)+\dots,\\
 d(e_2 \otimes e^1+ e_5 \otimes e^4)=-2e_6\otimes e^1\wedge e^4+\dots,\\
d(e_1{\otimes}w_{i_1,\dots,i_q+2}+e_3\otimes
\xi)=e_4\otimes(e^3{\wedge}w_{i_1,\dots,i_q+2}-e^1{\wedge}\xi)+\dots=\\
=e_4\otimes(e^3{\wedge}w_{i_1,\dots,i_q+2}+e^2{\wedge}D_2D_{-1}\xi-dD_{-1}\xi)+\dots,\;\;
d\xi=e^2{\wedge}w_{i_1,\dots,i_q+2}.
\end{multline}
It is easy to see that in the decomposition
$e^3{\wedge}w_{i_1,\dots,i_q+2}+e^2{\wedge}D_2D_{-1}\xi$ we have
the only one monomial
$e^3{\wedge}e^{i_1}{\wedge}\dots{\wedge}e^{i_q}{\wedge}e^{i_q+1}{\wedge}e^{i_q+2}$
with neighboring three last superscripts and it follows that the
expression is cohomologous to $w_{3,i_1,\dots,i_q+2}$ if $i_1 > 3$
and it is cohomologous to zero if $i_1=3$.
\end{proof}

\begin{proposition}
The differential $d_5$ is non trivial in the following cases:
\begin{multline}
\label{d_5}
d_5\left(e_1 \otimes (e^3\wedge e^4-e^2\wedge
e^5)\right)=e_6\otimes w_{3,4,5}, \quad d_5\left(e_2 \otimes
(e^3\wedge e^4-e^2\wedge e^5)\right)=e_7\otimes w_{3,4,5}, \\
d_5\left(e_1 \otimes w_{3,i_2,i_3,\dots,i_q{+}2}\right)=
-e_6\otimes w_{3,5,i_2,i_3,\dots,i_q{+}2}, i_2 >5,\\
d_5\left(e_2 \otimes w_{3,4,i_3,\dots,i_q{+}2}\right)=e_7 \otimes
w_{3,4,5,i_3,\dots,i_q{+}2}, i_3 > 5.
\end{multline}
\end{proposition}

\begin{proof}
Again we will not prove all the formulas (\ref{d_5}), leaving the
rest of them to the reader. For instance
\begin{multline}
d(e_1{\otimes}(e^3{\wedge}e^4-e^2{\wedge}
e^5)+\frac{1}{2}e_3{\otimes} (e^4{\wedge} e^5{-}e^3{\wedge}
e^6{+}e^2{\wedge} e^7)+\frac{1}{2}e_4{\otimes}(e^4{\wedge} e^6
{-}2 e^3 {\wedge} e^7 {+}3
e^2{\wedge}e^8)+\\
+\frac{1}{2}e_5\otimes(e^5\wedge e^6-2e^3\wedge e^8 +5e^2 \wedge
e^9))=e_6\otimes(e^5{\wedge}e^3{\wedge}e^4-\frac{1}{2}e^2{\wedge}(e^4{\wedge}e^6{-}2e^3{\wedge}e^7)
-\frac{1}{2}e^2{\wedge}e^4{\wedge}e^6{-}d\xi)+\dots,\\
\xi=\frac{1}{2}(e^5{\wedge}e^7-e^4{\wedge}e^8-e^3{\wedge}e^9+6e^2{\wedge}e^{10}),\;
w_{3,4,5}=e^3{\wedge}e^4{\wedge}e^5-e^2{\wedge}e^4{\wedge}e^6+e^2{\wedge}e^3{\wedge}e^7.
\end{multline}
\end{proof}

\begin{proposition}
The differential $d_s, s \ge 6$ can be non trivial only at the
classes of the form $e_1\otimes w_I$ or $e_2\otimes w_J$ and only
in the following cases:
\begin{multline}
J=(j_1,\dots,j_q,j_q+1,j_q+2), j_1=3, j_2=4,\dots, j_{s-3}=s-1,\;\;q >s-3, \;j_q>s;\\
I=(i_1,\dots,i_q,i_q+1,i_q+2), \; i_1=3,i_{s-4}=s-1, \;\; q > s-4,\;i_q>s;\\
I=(i_1,\dots,i_q,i_q+1,i_q+2), \; i_1=3,i_{s-4}=s-2, \;\; q > s-3,\;i_q>s;\\
I=J=(3,4,5,\dots,s-1).\\
\end{multline}
\end{proposition}
\begin{proof}
We will sketch the proof of this proposition in the spirit of
previous two propositions. The idea is the same: one has to keep
an eye only on the monomials of the form $e^{i_1}{\wedge}\dots
{\wedge}e^{i_q}{\wedge}e^{i_q+1}{\wedge}e^{i_q+2}$ in the
decomposition of scalar cocycles.
\end{proof}
Now a few words about forms $\tilde D_{-1}^jw$. We give no
explicit expressions for them like in the case of ${\mathfrak
m}_0$. However it is possible to write them out. For instance one
can define $\tilde D_{-1}w_{i_1,\dots,i_q+2}$ in a following way:
$$
\tilde D_{-1}w_{i_1,\dots,i_q+2} =D_{-1}w_{i_1,\dots,i_q+2}
-\sum_{s=0}\frac{(s+1)}{2^s}\omega\left((D_1^2+D_2)^sD_1(e^{i_1}{\wedge}\dots{\wedge}e^{i_q}){\wedge}e^{i_q+2+s}{\wedge}e^{i_q+3+s}\right).
$$
Now the next step is to define a form $\tilde D_{-1}^2w$ such that
\begin{equation}
\label{tilde} d\tilde D_{-1}^2w=e^1\wedge \tilde D_{-1}w+e^2\wedge
w.
\end{equation}
On the right hand side we have a closed form and it has the form
$e^1\wedge \xi +e^2\wedge \eta$. Hence it is the exact form and we
will denote by $\tilde D_{-1}^2w$ an arbitrary form satisfying
(\ref{tilde}).

Now we come to the inductive procedure. Let us assume that we have
found forms $\tilde D_{-1}^{j-1}w$ and $\tilde D_{-1}^{j-2}w$ such
that
$$
d\tilde D_{-1}^{j-1}w=e^1\wedge \tilde D_{-1}^{j-2}w+e^2 \wedge
D_{-1}^{j-3}w, \; d\tilde D_{-1}^{j-2}w=e^1\wedge \tilde
D_{-1}^{j-3}w+e^2 \wedge D_{-1}^{j-4}w.
$$
Then it is easy to see that the form $e^1\wedge \tilde
D_{-1}^{j-1}w+e^2 \wedge \tilde D_{-1}^{j-2}w$ is closed and hence
it is exact. We can define $D_{-1}^{j}w$ such that
$$
d\tilde D_{-1}^{j}w=e^1\wedge \tilde D_{-1}^{j-1}w+e^2 \wedge
D_{-1}^{j-2}w.
$$
Now it is evident that the element
$$
\Phi_{J,r}=\sum_{j=0}^{+\infty}e_{r+j}\otimes \tilde D_{-1}^jw_J,
\; J=(j_1,\dots,j_q,j_q+1,j_q+2),
$$
is closed element in $C^{q+2}({\mathfrak m}_0,{\mathfrak m}_0)$
and represents the class $e_{r}\otimes w_J$ in $E_1$ term.
\end{proof}

\begin{example}
The infinite basis of $H^2({\mathfrak m}_2,{\mathfrak m}_2)$
consists of the following cocycles:
\begin{multline}
\Phi_{2,3,m}=\Psi_{2,3,m},\;m=3, m\ge 7,\;
\Phi_{3,4,l}=\Psi_{3,4,l}, \; l \ge 3,\\
\Phi_{2,3,1}=e_1\otimes e^2{\wedge}
e^{3}+\frac{1}{2}\Psi_{4,5,5},\;
\Phi_{2,3,2}=\Psi_{2,3,2}+\frac{1}{2}\Psi_{4,5,6}.
\end{multline}
It was established by Vergne that an arbitrary two-dimensional
adjoint cocycle of ${\mathfrak m}_2$ vanishing at $e_1$ is
determined by its values on pairs $e_2, e_3$ and $e_3, e_4$
respectively \cite{V}. But we see that some of them, namely
$\Psi_{2,3,m},\;m=4,5,6,$ are coboundaries.
\end{example}

\begin{example}
The infinite dimensional $H^3({\mathfrak m}_0,{\mathfrak m}_0)$ is
the space of formal series of basic cocycles
$$
\Phi_{(i,i+1,i+2),r}=e_r \otimes w_{i,i+1,i+2}+e_{r+1}\otimes
\tilde D_{-1}w_{i,i+1,i+2}+\dots, r \ge 3, \; 3 \le i,
$$
where we have removed the cocycles $\Phi_{(3,4,5),6}$ and
$\Phi_{(3,4,5),7}$ from the list.
\end{example}

\begin{corollary}
1) $H^0(\mathfrak{m}_2,\mathfrak{m}_2)=0$.

2) $H^1(\mathfrak{m}_2,\mathfrak{m}_2)$ is infinite dimensional
and
$$  ~~{\rm dim} H_q^1(\mathfrak{m}_2,\mathfrak{m}_2)=
\left\{\begin{array}{r}
   0, \; q \le -1 \quad {\rm or}\quad q=1,\\
   1, \quad q \ge 2 \quad {\rm or} \quad q=0.
   \end{array} \right . \hspace{3.3em} $$

3) $H^2(\mathfrak{m}_2,\mathfrak{m}_2)$ is infinite dimensional
and
$$  ~~{\rm dim} H_q^2(\mathfrak{m}_2,\mathfrak{m}_2)=
\left\{\begin{array}{r}
   0, \; q \le -5,\\
   1, \; q=-1,0,1, \\
   2, \; q \ge 2 \quad {\rm or} \quad q=-4,-3,-2. \\
   \end{array} \right . \hspace{3.3em} $$
\end{corollary}
\begin{remark}
One-dimensional cohomology $H^1(\mathfrak{m}_2,\mathfrak{m}_2)$
was also found in \cite{V, GKh1, Kh} and rediscovered later in
\cite{FialWag2}. The property of $\mathfrak{m}_2$ that
$H^2_q(\mathfrak{m}_2,\mathfrak{m}_2)=0, q \le -5$ was established
in \cite{FialWag}. We corrected in the present article the values
of dimensions $\dim H^2_q(\mathfrak{m}_2,\mathfrak{m}_2)$ from
\cite{FialWag2} for $q=-1,0,1$: they are all equal to one.
\end{remark}

\end{document}